\documentclass{article}
\usepackage{graphicx}
\usepackage{amsmath}

\newtheorem{theorem}{Theorem}

\newenvironment{proof}[1][Proof]{\textbf{#1.} }{\ \rule{0.5em}{0.5em}}
\input{tcilatex}
\begin{document}

\author{Marek Galewski \\
Institute of Mathematics, \\
Technical University of Lodz,\\
Wolczanska 215, 90-924 Lodz, Poland, \\
marek.galewski@p.lodz.pl}
\title{A note on the dependence on parameters for a nonlinear system via
monotonicity theory}
\maketitle

\begin{abstract}
Using monotonicity theory we investigate the continuous dependence on
parameters for the discrete BVPs which can be written in a form of a
nonlinear system.
\end{abstract}

\section{Introduction}

In this note, which somehow completes \cite{galewskiCOERCIVE}, using
monotonicity theory we investigate the dependence on parameters for such
discrete boundary value problems which can put in a form of a so called
nonlinear system. Let $M$ be a metric space. We are interested in systems%
\begin{equation}
Ax=\lambda h(x,u),\text{ }x\in R^{n},\text{ }u\in M,  \label{non_sys}
\end{equation}%
where $\lambda >0$; $A$ is a $n\times n$ symmetric matrix; $h:R^{n}\times
M\rightarrow R^{n}$ is a continuous function; $u\in M$ serves as a
parameter. We assume that the function $h$ has a following form $h(x,u)=%
\left[ h_{1}\left( x_{1},u\right) ,h_{2}\left( x_{2},u\right)
,...,h_{n}\left( x_{n},u\right) \right] $, where $h_{i}:R\times M\rightarrow
R$ are continuous for $i=1,2,...,n.\bigskip $

Such systems arise often when a physical model undergoes discretization.
Indeed, let us for simplicity consider the following difference equation for 
$k\in \left\{ 1,2,...,n\right\} $ 
\begin{equation}
\Delta ^{2}x\left( k-1\right) =\lambda f\left( k,x\left( k\right) ,u\left(
k\right) \right) \text{, }x\left( 0\right) =x(n+1)=0  \label{problem_zintor}
\end{equation}%
with a positive parameter $\lambda $ and a continuous nonlinearity $%
f:\left\{ 1,2,...,n\right\} \times R\times R\rightarrow R$, $u$ stands for a
parameter. Such type of a difference equation may arise from evaluating the
Dirichlet boundary value problem 
\begin{equation*}
\frac{d^{2}}{dt^{2}}x=\lambda g\left( t,x,u\right) \text{, }0<t<1\text{, }%
x\left( 0\right) =x(1)=0
\end{equation*}%
where $g:\left[ 0,1\right] \times R\times R\rightarrow R$ is continuous and $%
f=g$ after changing of scale. With a positive definite $n\times n$ real
symmetric matrix 
\begin{equation*}
B=\left[ 
\begin{array}{ccccc}
2 & -1 & 0 & ... & 0 \\ 
-1 & 2 & -1 & ... & 0 \\ 
... & ... & ... & ... & ... \\ 
0 & ... & -1 & 2 & -1 \\ 
0 & ... & 0 & -1 & 2%
\end{array}%
\right] _{n\times n}
\end{equation*}%
and with $f(x,u)=\left[ f\left( 1,x\left( 1\right) ,u\left( 1\right) \right)
,...,f\left( n,x\left( n\right) ,u\left( n\right) \right) \right] $ problem (%
\ref{problem_zintor}) can be written into a form $Bx=\lambda f(x,u)$ for $%
x,u\in R^{n}.$ \bigskip

The existence of solutions and their multiplicity for nonlinear systems
received some attention lately, see \cite{YanZhangAPMacc} and references
therein. No results in the literature however concern the dependence on
parameters for a nonlinear system. The question whether the system depends
continuously on a parameter is vital in context of the applications, where
the measurements are known with some accuracy. Therefore, in the boundary
value problems for differential equations there are some results towards the
dependence of a solution on a functional parameter, see \cite%
{LedzewiczWalczak} with references therein. This is not the case with
discrete equations where we have only some results which use the critical
point theory, see \cite{galewskiCOERCIVE}. The approach of this note is
different from this of \cite{galewskiCOERCIVE}. While in \cite%
{galewskiCOERCIVE} the dependence on parameters is investigated through the
appropriate action functional, in the present note we investigate the
nonlinear system itself. An approach presented here is only applicable when
solutions are unique and therefore does not cover problems investigated in 
\cite{galewskiCOERCIVE}. Moreover, now we have included the dependence on a
parameter from some metric space.

\section{Continuous dependence on parameters}

We will fix two sets of assumptions depending on whether $f$ is super- or
sub-quadratic.

\begin{description}
\item[A1] \textit{There exist constants }$\gamma >2$, \textit{and }$\zeta ,$ 
$\theta >0$ \textit{such that }$\sum_{k=1}^{n}h_{k}(z_{k},v)z_{k}\geq \zeta
|z|_{R^{n}}^{\gamma }$\textit{\ for }$z=\left( z_{1},z_{2},...,z_{n}\right)
\in R^{n}$\textit{\ with }$\left\vert z\right\vert _{R^{n}}\geq \theta $%
\textit{\ and for all }$v\in M;$

\item[\textbf{A2}] \textit{there exists a constant }$a>0$\textit{\ such that}%
\begin{equation*}
\left( h_{k}(z_{1},v)-h_{k}(z_{2},v)\right) (z_{1}-z_{2})\geq
a|z_{1}-z_{2}|^{2}
\end{equation*}%
\textit{for each }$k=1,2,...,n,$\textit{\ for all }$z_{1},z_{2}\in R$\textit{%
\ and for all }$v\in M.$
\end{description}

Assumption \textbf{A1} is required so that to obtain the continuous
dependence on parameters and it does not influence the existence for which 
\textbf{A2} suffice.

\begin{theorem}
\label{TwDepPam}Assume \textbf{A1}, \textbf{A2 }and that $h\left( 0,u\right)
\neq 0$ for all $u\in M$. \newline
For any fixed parameter $u\in M$ and for all $\lambda \in (\frac{\left\Vert
A\right\Vert }{a},\infty )$ problem (\ref{non_sys}) have a nontrivial
solution in $R^{n}$. \newline
Let $\left\{ u_{k}\right\} _{k=1}^{\infty }\subset M$ be a convergent
sequence of parameters with $\lim_{k\rightarrow \infty }u_{k}=\overline{u}%
\in M$ and let $\lambda \in (\frac{\left\Vert A\right\Vert }{a},\infty )$ be
fixed. For any sequence $\left\{ x_{k}\right\} _{k=1}^{\infty }$ of
nontrivial solutions to (\ref{non_sys}) corresponding to $\left\{
u_{k}\right\} _{k=1}^{\infty }$, there exist a subsequence $\left\{
x_{k_{i}}\right\} _{i=1}^{\infty }\subset R^{n}$ and an element $\overline{x}%
\in R^{n}$, that $\lim_{i\rightarrow \infty }x_{k_{i}}=\overline{x}$.
Moreover, $\overline{x}$ satisfies problem (\ref{non_sys}) with $u=\overline{%
u}$, i.e. $A\overline{x}=\lambda h\left( \overline{x},\overline{u}\right) .$
\end{theorem}

\begin{proof}
Let us fix $u\in M$ and define an operator $K:R^{n}\rightarrow R^{n}$ by $%
K\left( x\right) =\lambda h\left( x,u\right) -Ax$. Then for $x,y\in R^{n}$
we have using \textbf{A2} 
\begin{equation*}
\begin{array}{l}
(Kx-Ky,x-y)_{R^{n}}=\lambda (h(x,u)-h(y,u),x-y)_{R^{n}}-(Ax-Ay,x-y)_{R^{n}}
\\ 
\geq \left( \lambda a-\left\Vert A\right\Vert \right) |x-y|_{R^{n}}^{2}\text{%
.}%
\end{array}%
\end{equation*}%
Hence $K$ is a strongly monotone continuous operator as long as $\lambda >%
\frac{\left\Vert A\right\Vert }{a}$ and so by the strongly monotone
principle, see \cite{deimling}, equation $Kw=0$ has a unique solution $%
w^{\ast }\in R^{n}$. Hence (\ref{non_sys}) has a nontrivial solution.

Corresponding to $\left\{ u_{k}\right\} _{k=1}^{\infty }$, there exists a
sequence $\left\{ x_{k}\right\} _{k=1}^{\infty }$ of solutions to (\ref%
{non_sys}) and also there exists exactly one $\overline{x}\in R^{n}$ such
that $A\overline{x}=h(\overline{x},\overline{u}).$

Now, we show that $\left\{ x_{k}\right\} _{k=1}^{\infty }$ is bounded.
Indeed, take an equation 
\begin{equation}
Ax=h(x,u_{k})  \label{pomequ}
\end{equation}%
with a fixed $u_{k}$. Let $x_{k}$ denotes its (unique) solution. If $%
\left\vert x_{k}\right\vert _{R^{n}}\leq \theta $ there is nothing to prove.
Otherwise, we multiply both sides of (\ref{pomequ}) with $x=x_{k}$ by $x_{k}$%
. Using \textbf{A1} and the inequality $\left( Ax_{k},x_{k}\right)
_{R^{n}}\leq \left\Vert A\right\Vert |x_{k}|_{R^{n}}^{2}$, we have 
\begin{equation}
\zeta |x_{k}|_{R^{n}}^{\gamma -2}\leq \left\Vert A\right\Vert .
\label{estim_s}
\end{equation}%
Hence a sequence $\left\{ x_{k}\right\} _{k=1}^{\infty }$ is bounded and it
has a convergent subsequence $\left\{ x_{k_{i}}\right\} _{i=1}^{\infty }$;
suppose element $\widetilde{x}\neq \overline{x}$ to be its limit. By
continuity letting $i\rightarrow \infty $ in both sides of $%
Ax_{k_{i}}=h(x_{k_{i}},u_{k_{i}})$ we get the conclusion that $A\widetilde{x}%
=h(\widetilde{x},\overline{u})$. Now by the uniqueness of solution $%
\widetilde{x}=\overline{x}$ and the assertion follows.
\end{proof}

Now we consider the subquadratic case assuming that

\begin{description}
\item[A3] \textit{There exist constants }$\mu <2$, $\theta _{1},\nu >0$%
\textit{\ such that }$\sum_{k=1}^{n}h_{k}(z_{k},v)z_{k}\leq \nu
|z|_{R^{n}}^{\mu }$\textit{\ for }$z=\left( z_{1},z_{2},...,z_{n}\right) \in
R^{n}$\textit{\ with }$\left\vert z\right\vert _{R^{n}}\geq \theta _{1}$%
\textit{\ and for all }$v\in M;$

\item[\textbf{A4}] \textit{there exist a constant }$b>0$\textit{\ such that }%
\begin{equation*}
\left( h_{k}(z_{1},v)-h_{k}(z_{2},v)\right) (z_{1}-z_{2})\leq
b|z_{1}-z_{2}|^{2}\mathit{\ }
\end{equation*}%
\textit{for each }$k=1,2,...,n,$\textit{\ for} \textit{all }$z_{1},z_{2}\in
R $\textit{\ and for all }$v\in M.$
\end{description}

Let us denote by $\lambda _{1}\leq \lambda _{2}\leq ...\leq \lambda _{n}$
the eigenvalues of $A$.

\begin{theorem}
Assume \textbf{A3}, \textbf{A4 }and that $h\left( 0,u\right) \neq 0$ for all 
$u\in M$. Let moreover $A$ be positive definite. Than all assertions of
Theorem \ref{TwDepPam} hold with $\lambda \in \left( 0,\frac{\lambda _{1}}{b}%
\right) $.
\end{theorem}

\begin{proof}
The proof is similar to that of Theorem \ref{TwDepPam} apart from the
existence part and the counterpart of (\ref{estim_s}). We define a
continuous operator $K_{1}:R^{n}\rightarrow R^{n}$ by $Kx=Ax-\lambda h(x,u)$
and we fix $u\in M$. We get for $x,y\in R^{n}$ that 
\begin{equation*}
(Ax-Ay,x-y)_{R^{n}}-\lambda (h(x,u)-h(y,u),x-y)_{R^{n}}\geq \left( \lambda
_{1}-\lambda b\right) |x-y|^{2}\text{.}
\end{equation*}%
\ Hence $K_{1}$ is strongly monotone operator as long as $\lambda \in (0,%
\frac{\lambda _{1}}{b})$.

In order to obtain a counterpart of (\ref{estim_s}) we multiply both sides
of (\ref{pomequ}) with $x=x_{k}$ by $x_{k}$. Using \textbf{A3} and the
inequality $\left( Ax_{k},x_{k}\right) _{R^{n}}\geq \lambda
_{1}|x_{k}|_{R^{n}}^{2}$, we have $\nu |x_{k}|_{R^{n}}^{\mu -2}\geq \lambda
_{1}|x_{k}|_{R^{n}}^{2}.$
\end{proof}

\section{Final comments}

Any discrete boundary value problem which can be considered as a nonlinear
system can be investigated by our approach. As an example we may mention the
discrete version of the Emden-Fowler equation which originated in the
gaseous dynamics in astrophysics. This discretization received some
attention lately mainly by the use of critical point theory, see for
example, \cite{HeWe}, \cite{AppMathLett}, with results pertaining to the
existence of solutions and their multiplicity. In \cite{HeWe} the discrete
Emden-Fowler equation is considered in a form of a nonlinear system and
therefore falls within the framework which we just described.

\end{document}